\documentclass{amsart}
\let\\=\backslash
\let\sse=\subseteq
\let\noi=\noindent
\let\limply=\Longrightarrow

\def\0{\{0\}}
\def\codim{{\rm codim}\,}
\def\span{{\kern.5pt{\rm span}\kern1pt}}
\def\conv{{\;\longrightarrow\;}}
\def\sconv{{{\buildrel_{\scriptstyle s}\over\conv}}}
\def\uconv{{{\buildrel_{\scriptstyle u}\over\conv}}}

\def\B{{\mathcal B}}
\def\F{{\mathcal F}}
\def\M{{\mathcal M}}
\def\N{{\mathcal N}}
\def\R{{\mathcal R}}
\def\U{{\mathcal U}}
\def\V{{\mathcal V}}
\def\X{{\mathcal X}}
\def\Y{{\mathcal Y}}

\def\BX{{\B[\X]}}
\def\BY{{\B[\Y\kern1pt]}}

\newtheorem{theorem}{Theorem}
\newtheorem{lemma}{Lemma}
\newtheorem{corollary}{Corollary}
\newtheorem{proposition}{Proposition}

\theoremstyle{definition}
\newtheorem{definition}{Definition}
\newtheorem{remark}{Remark}

\numberwithin{theorem}{section}
\numberwithin{lemma}{section}
\numberwithin{corollary}{section}
\numberwithin{proposition}{section}
\numberwithin{conjecture}{section}
\numberwithin{definition}{section}
\numberwithin{remark}{section}
\numberwithin{question}{section}

\begin{document}

\vglue-34pt\noindent
\hfill{\it Studia Scientiarum Mathematicarum Hungarica}\/,
{\bf 55} (2018) 327--344

\vglue20pt
\title{Range-Kernel Complementation}
\author{C.S. Kubrusly}
\address{Applied Mathematics Department, Federal University,
         Rio de Janeiro, RJ, Brazil}
\email{carloskubrusly@gmail.com}
\renewcommand{\keywordsname}{Keywords}
\date {July 24, 2016}

\begin{abstract}
If a Banach-space operator has a complemented range, then its normed-space
adjoint has a complemented kernel and the converse holds on a reflexive
Banach space$.$ It is also shown when complemented kernel for an operator
is equivalent to complemented range for its normed-space adjoint$.$ This is
applied to compact operators and to compact perturbations$.$ In particular,
compact perturbations of semi-Fredholm operators have complemented range
and kernel for both the perturbed operator and its normed-space adjoint.
\end{abstract}

\maketitle

\vskip-10pt\noi
\section{Introduction}

A subspace (i.e., a closed linear manifold) of a normed space is
complemented if it has a subspace as an algebraic complement$.$ In a
Hilbert space every subspace is complemented, which is not the case in a
Banach space$.$ Banach-space operators with complemented range and kernel
play a crucial role in many aspects of operator theory, especially in
Fredholm theory$.$ In fact, range-kernel complementation is the main issue
behind the difference between the Hilbert-space and the Banach-space
approaches for dealing with Fredholm operators \cite{KD2}$.$ Precisely,
range-kernel complementation is what differentiates the upper-lower approach
and the left-right approach for investigating semi-Fredholm operators$.$ In
particular, it has recently been shown how such a difference, based on the
notion of range-kernel complementation, leads to the characterization of
biquasitriangular operators on a Banach space \cite[Section 3]{KD1},
\cite[Section 6]{KD2}.

\vskip6pt
The main result of this paper is stated in Theorem 3.1 which addresses to
the question on how complementedness for range and kernel travels from an
operator to its adjoint and back, without assuming a priory that the
operator has closed range$.$ This is applied in Corollaries 5.1 to 5.4
towards compact operators and compact perturbations.

\vskip6pt
The paper is organized as follows$.$ Section 2 sets up notation and
terminology, including the concepts of upper-lower and left-right
semi-Fredholmness$.$ Section 3 reports on the difference between the
upper-lower and the left-right approaches to semi-Fredholm operators in
terms of range-kernel complementation, which is stated in Proposition 3.1$.$
All propositions in Sections 3, 4 and 5 are well-known results which are
applied throughout the text, some of them are used quite frequently and so
they are stated in full (whose proofs are always addressed to current
literature)$.$ Section 3 closes with the statement of Theorem 3.1 which
reads as follows$.$ If an operator has a complemented range then its
normed-space adjoint has a complemented kernel and the converse holds on a
reflexive Banach space; dually, it also shows when complemented kernel for
an operator is equivalent to complemented range for its normed-space
adjoint$.$ Section 4 deals with the proof of Theorem 3.1 which is based on
Lemmas 4.1 and 4.2$.$ Applications are considered in Section 5 where
Corollary 5.1 shows when compact operators and their normed-space adjoints
have complemented range and kernel$.$ Corollary 5.2 is split into two parts,
where Part 1 presents a rather simplified proof of a previous result from
\cite{Hol} on compact perturbations of bounded below operators$.$ In
addition its is shown that these (and they normed-space adjoints) have
complemented range and kernel, which is translated to semi-Fredholm
operators in Corollaries 5.3 and 5.4.

\vskip0pt\noi
\section{Notation and Terminology}

Let $\X$ be a linear space and let $\M$ be any linear manifold of $\X.$
Every linear manifold is complemented in the sense that it has a linear
manifold as an algebraic complement$.$ That is, for every linear manifold
$\M$ there is another linear manifold $\N$ for which $\X={\M+\N}$ and
${\M\cap\N}=\0$, where $\N$ is referred to as an algebraic complement of
$\M$ (and vice versa)$.$ Let ${\X/\M}$ stand for the quotient space of $\X$
modulo $\M$ (i.e., the linear space of all cosets ${[x]=x+\M}$ of $x$ modulo
$\M$)$.$ Recall: ${\codim\M=\dim\X/\M}$, where codimension of $\M$ means
dimension of any algebraic complement of $\M.$ If ${L\!:\X\!\to\X}$ is a
linear transformation of a linear space $\X$ into itself, then let
$\N(L)=L^{-1}(\0)$ and $\R(L)=L(\X)$ denote the kernel and range of $L$,
respectively, which are linear manifolds of $\X.$ A projection is an
idempotent (i.e., ${E=E^2}$) linear transformation ${E\!:\X\!\to\X}$ of a
linear space $\X$ into itself, and ${I-E\!:\X\!\to\X}$ is the complementary
projection of $E$, where $\N({I-E})=\R(E)$ and $\R({I-E})=\N(E)$, with
${I\!:\X\to\X}$ standing for the identity transformation.

\vskip6pt
Suppose $\X$ is a normed space$.$ By a subspace of $\X$ we mean a closed
(in the norm topology of $\X$) linear manifold of $\X.$ Let $\M^-\!$
denote the closure (in the norm topology of $\X$) of a linear manifold
$\M$ of $\X$, which is a subspace of $\X.$ Let $\BX$ stand for the normed
algebra of all operators on $\X$, which means of all bounded linear (i.e.,
continuous linear) transformations of $\X$ into itself$.$ The kernel
$\N(T)$ of any operator ${T\in\BX}$ is a subspace (i.e., a closed linear
manifold) of $\X$.

\vskip2pt\noi
\begin{definition}
(See, e.g., \cite[Definition 16.1]{Mul})$.$
Let $\X$ be a Banach space and
consider the following classes of operators on $\X$.
$$
\Phi_+[\X]
=\big\{T\in\BX\!:\,\R(T)\;\hbox{is closed and}\;\dim\N(T)<\infty\big\},
$$
the class of {\it upper semi-Fredholm}\/ operators from $\BX$, and
$$
\Phi_-[\X]
=\big\{T\in\BX\!:\,\R(T)\;\hbox{is closed and}\;\dim\X/\R(T)<\infty\big\},
$$
the class of {\it lower semi-Fredholm}\/ operators from $\BX.$ Set
$$
\Phi[\X]=\Phi_+[\X]\cap\Phi_-[\X],
$$
which is the class of {\it Fredholm}\/ operators from $\BX$.
\end{definition}

\vskip0pt\noi
\begin{definition}
(See, e.g., \cite[Section 5.1]{ST})$.$ Let $\X$ be a Banach space and
consider the following classes of operators on $\X$.
\begin{eqnarray*}
\F_\ell[\X]\!
&\kern-6pt=\kern-6pt&
\!\!\big\{T\!\in\!\BX\!:T\;\hbox{is left essentially invertible}\big\}  \\
&\kern-6pt=\kern-6pt&
\!\!\big\{T\!\in\!\BX\!:S\kern1ptT=I+K\;\hbox{for some}\;S\in\BX
\;\hbox{and some compact}\;K\in\BX\big\}
\end{eqnarray*}
is the class of {\it left semi-Fredholm}\/ operators from $\BX$, and
\begin{eqnarray*}
\F_r[\X]\!
&\kern-6pt=\kern-6pt&
\!\!\big\{T\!\in\!\BX\!:T\;\hbox{is right essentially invertible}\big\} \\
&\kern-6pt=\kern-6pt&
\!\!\big\{T\!\in\!\BX\!:TS=I+K\;\hbox{for some}\;S\in\BX
\;\hbox{and some compact}\;K\in\BX\big\}
\end{eqnarray*}
is the class of {\it right semi-Fredholm}\/ operators from $\BX.$ Set
$$
\F[\X]=\F_\ell[\X]\cap\F_r[\X]
=\big\{T\in\BX\!:T\;\hbox{is essentially invertible}\big\},
$$
which the class of {\it Fredholm}\/ operators from $\BX.$
\end{definition}

\vskip4pt
For a collection of relations among $\Phi_+[\X]$, $\,\Phi_-[\X]$,
$\,\F_\ell[\X]$, and $\F_r[\X]$ see, e.g., \cite[Section 3]{KD2}$.$ In
particular,
$$
\Phi[\X]=\F[\X].
$$
The classes $\Phi_+[\X]$ and $\Phi_-[\X]$ are open in $\BX$ (see, e.g.,
\cite[Proposition 16.11]{Mul}), and so are the classes $\F_\ell[\X]$ and
$\F_r[\X]$ (see e.g., \cite[Proposition XI.2.6]{Con}).

\vskip0pt\noi
\section{Range-Kernel Complementation}

For subspaces, the definition of complementation reads as follows$.$ A
subspace $\M$ of a normed space $\Y$ is {\it complemented}\/ if it has a
subspace as an algebraic complement$.$ In other words, a {\it closed}\/
linear manifold $\M$ of a normed space $\Y$ is complemented if there is a
{\it closed}\/ linear manifold $\N$ of $\Y$ such that $\M$ and $\N$ are
algebraic complements. In other words,
$$
\M+\N=\Y
\quad\;\hbox{and}\;\quad
\M\cap\N=\0.
$$
In this case $\M$ and $\N$ are {\it complementary subspaces}\/ --- one is
the {\it complement}\/ of the other$.$ (Closedness here refers to the norm
topology of $\Y$)$.$ A normed space is {\it complemented}\/ if every
subspace of it is complemented$.$ If a Banach space is complemented, then
it is isomorphic (i.e., topologically isomorphic) to a Hilbert space
\cite{LT} (see also \cite{Kal})$.$ Thus complemented Banach spaces are
identified with Hilbert spaces --- only Hilbert spaces (up to an
isomorphism) are complemented.

\vskip2pt\noi
\begin{definition}
Let $\Y$ be a normed space$.$ Define the following classes of
operators on $\Y$.
\vskip-2pt\noi
$$
\Gamma_R[\Y]=\big\{S\in\BY\!:\,\R(S)^-
\;\hbox{is a complemented subspace of $\Y$}\big\},
$$
$$
\Gamma_N[\Y]=\big\{S\in\BY\!:\,\N(S)
\;\hbox{is a complemented subspace of $\Y$}\big\},
$$
where $\R(S)^-$ stands for the closure of $\R(S)$ in the norm topology
of $\Y$.
\end{definition}

\vskip2pt
Left and right and upper and lower semi-Fredholm operators are linked
by range and kernel complementation$:$ ${T\in\F_\ell[\X]}$ if and only
if ${T\in\Phi_+[\X]}$ and $\R(T)$ is complemented, and ${T\in\F_r[\X]}$
if and only if ${T\in\Phi_-[\X]}$ and $\N(T)$ is complemented.

\vskip2pt\noi
\begin{proposition}
Let\/ $\X$ be a Banach space\/ $\X$.
\begin{eqnarray*}
\F_\ell[\X]
&\kern-6pt=\kern-6pt&
\Phi_+[\X]\cap\Gamma_R[\X]                                              \\
&\kern-6pt=\kern-6pt&
\big\{T\in\Phi_+[\X]\!:\,\R(T)
\;\hbox{\rm is a complemented subspace of $\X$}\big\}.
\end{eqnarray*}
\vskip-6pt\noi
\begin{eqnarray*}
\F_r[\X]
&\kern-6pt=\kern-6pt&
\Phi_-[\X]\cap\Gamma_N[\X]                                              \\
&\kern-6pt=\kern-6pt&
\big\{T\in\Phi_-[\X]\!:\,\N(T)
\;\hbox{\rm is a complemented subspace of $\X$}\big\}.
\end{eqnarray*}
\end{proposition}

\begin{proof}
\cite[Theorems 16.14, 16.15]{Mul} (since ${\R(T)^-\kern-3pt=\kern-1pt\R(T)}$
if\/ ${T\!\in\kern-1pt\Phi_+[\X]\cup\Phi_-[\X]}).\!$
\end{proof}

\vskip2pt
In particular, if a Banach space $\X$ is complemented (i.e., if $\X$ is
isomorphic to a Hilbert space), then Proposition 3.1 says
$$
\Phi_+[\X]=\F_\ell[\X]
\quad\;\hbox{and}\;\quad
\Phi_-[\X]=\F_r[\X].
$$
Conditions leading to the above identities have been considered in
\cite[Lemma 3.1 and Remark 3.1]{KD2}$.$ In such a particular case the
upper-lower and the left-right approaches for dealing with semi-Fredholm
operators coincide (as it is always case in a Hilbert space setting).

\vskip6pt
Throughout the text let $\X^*$ denote for the dual space of a normed space
$\X$ (and so $\X^*$ is a Banach space), and let ${T^*\in\B[\X^*]}$ stand
for the normed-space adjoint of an operator ${T\in\BX}$ (see, e.g.,
\cite[Section 3.1]{Meg} or \cite[Section 3.2]{Sch}).

\vskip2pt\noi
\begin{remark}
As it is known
$$
T\in\Phi_+[\X]
\quad\iff\quad
T^*\in\Phi_-[\X^*],
$$
\vskip-2pt\noi
$$
T\in\Phi_-[\X]
\quad\iff\quad
T^*\in\Phi_+[\X^*].
$$
(See, e.g., \cite[Theorem 16.4]{Mul}.) Moreover,
$$
T\in\F_\ell[\X]
\quad\iff\quad
T^*\in\F_r[\X^*],
$$
\vskip-2pt\noi
$$
T\in\F_r[\X]\quad\iff\quad
T^*\in\F_\ell[\X^*].
$$
Actually, as it is readily verified the above equivalences hold in a
Hilbert space --- see, e.g., \cite[Section 2]{Kub} or
\cite[Section 5.1]{ST} --- whose proof extends naturally from Hilbert
spaces to reflexive Banach spaces by using standard properties of
normed-space adjoints (see, e.g.,
\cite[Propositions 3.1.4, 3.1.10, and 3.1.13]{Meg}) as well as Schauder
Theorem on compactness for normed-space adjoints of compact operators on
a Banach space --- see e.g., \cite[Theorem 3.4.15]{Meg}).
\end{remark}

\vskip2pt
Remark 3.1 motivates the question on how complementedness for range and
kernel travels from an operator to its normed-space adjoint$.$ In light
of Proposition 3.1 and Remark 3.1 we might expect
\vskip2pt\noi
\begin{description}
\item{$\kern-4pt$\rm(a)}
$\,{T\in\Gamma_R[\X]}$ if and only if ${T^*\in\Gamma_N[\X^*]}$,
\vskip2pt
\item{$\kern-4pt$\rm(b)}
$\,{T\in\Gamma_N[\X]}$ if and only if ${T^*\in\Gamma_R[\X^*]}$.
\end{description}
\vskip2pt\noi
Indeed, Proposition 3.1 and Remark 3.1 might suggest the above equivalence,
although when using them one is bound to assume a priory semi-Fredholmness,
thus one is bound to assume a priori operators with closed range$.$ In fact,
we show that (a) holds true without the closedness assumption (up to
reflexivity in one direction), and (b) also holds without the closedness
assumption (up to reflexivity in one direction) and, in the opposite
direction, it holds if $\R(T)$ is closed$.$ This is stated below
(Theorem 3.1) and is proved in Section 4 independently of its relationship
with semi-Fredholm operators$.$ That is, without using any properties of the
classes of upper-lower or left-right semi-Fredholm operators (where ranges
necessarily satisfy the particular assumption of closedness).

\vskip2pt\noi
\begin{theorem}
Let\/ $\X$ be a Banach space and take any operator\/ ${T\in\BX}$.
\vskip2pt\noi
\begin{description}
\item{$\kern-9pt$\rm(a$_1$)}
If\/ $\R(T)^-$ is complemented, then\/ $\N(T^*)$ is complemented\/:
$$
T\in\Gamma_R[\X]
\quad\limply\quad
T^*\in\Gamma_N[\X^*].
$$
\item{$\kern-9pt$\rm(a$_2$)}
If $\X$ is reflexive and $\N(T^*)$ is complemented, then $\R(T)^-$ is
complemented\/:
$$
\X
\;\hbox{\rm reflexive and}\;\;
T^*\in\Gamma_N[\X^*]
\quad\limply\quad
T\in\Gamma_R[\X].
$$
\item{$\kern-9pt$\rm(b$_1$)}
If\/ $\X$ is reflexive and $\R(T^*)^-$ is complemented, then\/ $\N(T)$
is complemented\/:
$$
\X
\;\hbox{\rm reflexive and}\;\;
T^*\in\Gamma_R[\X^*]
\quad\,\limply\quad
T\in \Gamma_N[\X].
$$
\item{$\kern-9pt$\rm(b$_2$)}
If\/ $\R(T)$ is closed and\/ $\N(T)$ is complemented then\/ $\R(T^*)$
is complemented\/:
$$
\kern20pt
\R(T)=\R(T)^-
\;\hbox{\rm and}\;\;
T\in \Gamma_N[\X]
\;\limply\;
\R(T^*)=\R(T^*)^-
\;\hbox{\rm and}\;\;
T^*\in\Gamma_R[\X^*].
$$
\end{description}
\end{theorem}

\vskip0pt\noi
\section{Proof of Theorem 3.1}

A subspace $\M$ of a normed space $\Y$ is {\it weakly complemented}\/ if it
is weakly closed (i.e., closed in the weak topology of $\Y$) and there
exists a weakly closed linear manifold $\N$ of $\Y$ (so that $\N$ is closed
in the norm topology of $\Y$, and hence $\N$ is a subspace of $\Y$) for
which
$$
\M+\N=\Y
\quad\;\hbox{and}\;\quad
\M\cap\N=\0.
$$
Similarly, if $\X$ is a normed space, then a subspace $\U\!$ of the Banach
space $\X^*$ (the dual of $\X$) is {\it weakly* complemented}\/ if it is
weakly* closed (i.e., closed in the weak* topology of $\X^*$) and there
exists a weakly* closed linear manifold $\V$ of $\X^*$ (so that $\V$ is
closed in the norm topology of $\X^*\!$, and hence $\V$ is a subspace of
$\X^*$) for which
$$
\U\!+\V=\X^*
\quad\;\hbox{and}\;\quad
\U\cap\V=\0.
$$
Thus weak complementation in $\Y$ (or weak* complementation in $\X^*$) is
obtained from (plain) complementation in $\Y$ (or in $\X^*$) if closeness
in the norm topology of $\Y$ (or in the norm topology of $\X^*$) is replaced
with closeness in the weak topology of $\Y$ (or with closeness in the
weak* topology of $\X^*\!$) for both complemented subspaces $\M$ and $\N$
(or $\U\!$ and $\V)$.

\vskip2pt\noi
\begin{definition}
Let $\Y$ be a normed space$.$ Define the following classes of operators
on $\Y$.
\vskip-2pt\noi
$$
w\hbox{-}\Gamma_R[\Y]=\big\{S\!\in\B[\Y]\!:\,
\hbox{$\R(S)^-$ is weak complemented in $\Y$}\big\},
$$
\vskip-2pt\noi
$$
w\hbox{-}\Gamma_N[\Y]=\big\{S\!\in\B[\Y]\!:\,
\hbox{$\N(S)$ is weak complemented in $\Y$}\big\}.
$$
\end{definition}

\vskip0pt\noi
\begin{definition}
Let $\X$ be a normed space$.$ Define the following classes of operators
on its dual $\X^*\!$.
$$
w^*\hbox{-}\Gamma_R[\X^*]=\big\{F\!\in\B[\X^*]\!:\,
\hbox{$\R(F)^-$ is weak* complemented in $\X^*$}\big\},
$$
$$
w^*\hbox{-}\Gamma_N[\X^*]=\big\{F\!\in\B[\X^*]\!:\,
\hbox{$\N(F)$ is weak* complemented in $\X^*$}\big\}.
$$
\end{definition}

\vskip4pt
Here $\R(S)^-$ and $\R(F)^-$ are the closures of the ranges in the norm
topology of $\Y$ or $\X^*\!$, which are now supposed to be, in addition,
closed in the weak topology of $\Y$ (Definition 4.1) or closed in the
weak* topology of $\X^*\!$ (Definition 4.2), as also are $\N(S)$ and
$\N(F)$, to meet the definitions of weak and weak* complementation.

\vskip2pt\noi
\begin{proposition}
\vskip6pt\noi
A convex set in a normed space is closed\/ $($in the norm topology\/$)$ if
and only if it is weakly closed\/ {\rm(i.e.}, closed in the weak
topology\/$).$ {\rm In particular,} a linear manifold of a normed space is
closed if and only if it is weakly closed.
\end{proposition}

\begin{proof}
See, e.g.,
\cite[Theorem V.1.4, Corollary V.1.5]{Con} or
\cite[Theorem 2.5.16, Corollary 2.5.17]{Meg}.
\end{proof}

\vskip0pt\noi
\begin{corollary}
If $\X$ is a normed space, then
$$
w\hbox{-}\Gamma_R[\X]=\Gamma_R[\X],
\qquad
w\hbox{-}\Gamma_R[\X^*]=\Gamma_R[\X^*].
$$
$$
w\hbox{-}\Gamma_N[\X]=\Gamma_N[\X],
\qquad
w\hbox{-}\Gamma_N[\X^*]=\Gamma_N[\X^*].
$$
\end{corollary}

\begin{proof}
Straightforward from Definitions 3.1 and 4.1, and Proposition 4.1.
\end{proof}

\vskip0pt\noi
\begin{proposition}
{\rm(a)} If\/ $\Y$ is a normed space and\/ ${E\!:\Y\to\Y}$ is a continuous
projection, then\/ $\R(E)$ and\/ $\N(E)$ are complementary subspaces of\/
$\Y.$ Conversely, {\rm(b)} if\/ $\M$ and\/ $\N$ are complementary subspaces
of a Banach space\/ $\Y$, then the\/ $($unique\/$)$ projection\/
${E\!:\Y\to\Y}$ with\/ $\R(E)=\M$ and\/ $\N(E)=\N$ is continuous $($and so
every complemented subspace of a Banach space is the range of a bounded
linear operator$)$.
\end{proposition}

\begin{proof}
See, e.g., \cite[Theorem 3.2.14 and Corollary 3.2.15]{Meg} or
\cite[Problem 4.35]{EOT}.
\end{proof}

\vskip0pt\noi
\begin{proposition}
Let\/ $\X$ be a normed space$.$ Suppose\/ ${T\!:\X\to\X}$ is linear.
\vskip2pt\noi
\begin{description}
\item{$\kern-4pt$\rm(a)}
$T$ is continuous 
$($in the norm topology of\/ $\X)$ if and only if it is weakly continuous
{\rm(i.e.,} continuous in the weak topology of\/ $\X)$; that is,
$$
T\in\BX \quad \iff \quad T
\hbox{ is weakly continuous.}
$$
\vskip2pt\noi
\item{$\kern-4pt$\rm(b)}
If\/ $T$ is continuous in the norm topology of\/ $\X$, then\/ $T^*\!$ is
continuous in the weak* topology of\/ $\X^*$; that is,
$$
T\in\BX \quad \limply \quad T^*\in\B[\X^*]
\hbox{ is weakly* continuous.}
$$
\end{description}
\vskip2pt\noi
Conversely, if
\vskip2pt\noi
\begin{description}
\item{$\kern-4pt$\rm(c$_1$)}
${F\!:\X^*\to\X^*}\!$ is linear and continuous in the weak* topology
of\/ $\X^*\!$,
\end{description}
then
\begin{description}
\item{$\kern-4pt$\rm(c$_2$)}
${F=S^*\!}$ for some operator ${S\in\BX}$,
\end{description}
which implies
\begin{description}
\item{$\kern-4pt$\rm(c$_3$)}
${F\in\B[\X^*]}$ 
\quad
{\rm(i.e.,} $F$ is linear and continuous in the norm topology of\/ $\X^*)$.
\end{description}
\end{proposition}

\begin{proof}
See, e.g., \cite[Theorem 2.5.11, Theorem 3.1.11 and Corollary 3.1.12]{Meg}.
\end{proof}

\vskip0pt\noi
\begin{proposition}
$\!$Let $\X\!$ be a Banach space$.$ The following assertions are equivalent.
\vskip2pt\noi
\begin{description}
\item{$\kern-4pt$\rm(a)}
$\X$ is reflexive.
\vskip2pt
\item{$\kern-4pt$\rm(b)}
$\X^*$ is reflexive.
\vskip2pt
\item{$\kern-4pt$\rm(c)}
The weak* and the weak topologies in\/ $\X^*$ coincide
\vskip0pt\noi
{\rm(i.e.,} $\sigma(\X^*\!,\X)=\sigma(\X^*\!,\X^{**})\,)$.
\end{description}
\end{proposition}

\begin{proof}
See, e.g., \cite[p$.$346, Problem 16.N]{BP} or \cite[Theorem V.4.2]{Con}.
\end{proof}

\vskip0pt\noi
\begin{proposition}
If\/ $\R(S)$ is the range of an operator\/ ${S\in\BX}$ on a Banach space\/
$\X$, then the following assertions are equivalent.
\vskip2pt\noi
\begin{description}
\item{$\kern-4pt$\rm(a)}
$\R(S)$ is closed in the norm topology of\/ $\X$.
\vskip2pt
\item{$\kern-4pt$\rm(b)}
$\R(S^*)$ is closed in the norm topology of\/ $\X^*\!$.
\vskip2pt
\item{$\kern-4pt$\rm(c)}
$\R(S^*)$ is closed in the weak* topology of\/ $\X^*\!$.
\end{description}
\end{proposition}

\begin{proof}
See, e.g.,\cite[Theorem VI.1.10]{Con} or \cite[Theorem 3.1.21]{Meg}.
\end{proof}

\vskip0pt\noi
\begin{lemma}
For every normed space $\X$,
\vskip6pt\noi
{\,\rm(i)}
$\;w^*\hbox{-}\Gamma_R[\X^*]\sse\Gamma_R[\X^*]$
\quad\;\hbox{and}\;\quad
$w^*\hbox{-}\Gamma_N[\X^*]\sse\Gamma_N[\X^*]$.
\vskip6pt\noi
{\rm (ii)}
$\,$If\/ $\X$ is a reflexive Banach space, then reverse inclusions hold:
\vskip2pt\noi
\begin{description}
\item{$\kern-4pt$\rm(a)}
$\;\Gamma_R[\X^*]\sse w^*\hbox{-}\Gamma_R[\X^*]$,
\vskip2pt
\item{$\kern-4pt$\rm(b)}
$\;\Gamma_N[\X^*]\sse w^*\hbox{-}\Gamma_N[\X^*]$.
\end{description}
\end{lemma}

\begin{proof}
(i) The inclusions
$$
w^*\hbox{-}\Gamma_R[\X^*]\sse\Gamma_R[\X^*]
\quad\;\hbox{and}\;\quad
w^*\hbox{-}\Gamma_N[\X^*]\sse\Gamma_N[\X^*]
$$
hold since weak* complementation implies complementation in the norm
topology$.$ In fact, weak* closedness implies closedness in the norm
topology --- reason: weak* topology in $\X^*$ is weaker than the weak
topology in $\X^*\!$, which in turn is weaker than the norm topology in
$\X^*\!.$ In other words,
$\sigma(\X^*\!,\X)\sse\sigma(\X^*\!,\X^{**})\sse\sigma(\X^*\!,\X^*)$.

\vskip6pt
\hskip20pt
(ii) To prove the reverse inclusions, proceed as follows.

\vskip6pt\noi
(a) $\Gamma_R[\X^*]\sse w^*\hbox{-}\Gamma_R[\X^*].$ Indeed, suppose
${F\in\Gamma_R[\X^*]}.$ So the subspace $\R(F)^-$ is complemented in
$\X^*\!.$ That is, there is a subspace $\V$ (a linear manifold of
$\X^*$ closed in the norm topology of $\X^*\!$) for which
$\X^*\!={\R(F)^-\!+\V}$ and ${\R(F)^-\!\cap\V}=\0.\!$ Thus both subspaces
$\R(F)^-\!$ and $\V$ are complemented in the Banach space $\X^*\!$, and this
implies both $\R(F)^-\!$ and $\V$ are ranges of operators in $\B[\X^*]$;
precisely, there are operators $P$ and $P'={I-P}$ in $\B[\X^*]$ (with closed
range) for which $\R(F)^-\!=\R(P)$ and $\V=\R(P')$ --- cf$.$
Proposition 4.2(b)$.$ Since ${P,P'\in\B[\X^*]}$ are continuous in the norm
topology of $\X^*$, they are continuous in the weak topology of $\X^*$ by
Proposition 4.3(a)$.$ Thus if $\X$ is reflexive, then by
Proposition 4.4(a,c) $P$ and $P'$ on $\X^*$ are weakly* continuous, and
therefore $P$ and $P'$ are normed-space adjoints of some operators
${S,S'\in\BX}$ according to Proposition 4.3(c); that is, ${P=S^*}$ and
${P'={S'}^*}\!.$ Recall: $\R(P)$ and $\R(P')$ are closed in the norm
topology of $\X^*\!$, and so $\R(S^*)$ and $\R({S'}^*)$ are closed in the
norm topology of $\X^*\!.$ Since $\R(F)^-\!=\R(P)=\R(S^*)$ and
$\V=\R(P')=\R({S'}^*)$, the subspaces $\R(F)^-\!$ and $\V$ are closed ranges
of the normed-space adjoints $S^*$ and ${S'}^*$ of operators $S$ and $S'$
(closed in the norm topology of $\X^*$), and therefore they are closed in
the weak* topology of $\X^*\!$ according to Proposition 4.5(b,c)$.$ Then
${F\!\in w^*\hbox{-}\Gamma_R[\X^*]}.$ Therefore
$\Gamma_R[\X^*]\sse w^*\hbox{-}\Gamma_R[\X^*]$.

\vskip6pt\noi
(b) $\Gamma_N[\X^*]\sse w^*\hbox{-}\Gamma_N[\X^*].$ In fact,
suppose ${F\in \Gamma_N[\X^*]}.$ Thus the subspace $\N(F)$ is complemented
in $\X^*\!.$ That is, there exists another subspace $\V$ of $\X^*$ for which
$\X={\N(F)+\V}$ and ${\N(F)\cap\V}=\0.$ Hence $\N(F)$ and $\V$ are
complemented subspaces of the Banach space $\X^*\!$, and so they are
(closed) ranges of normed-space adjoints of some operators in $\B[\X]$,
and therefore the subspaces $\N(F)$ and $\V$ are weakly* closed by using
the same argument as in item (a)$.$ Then
${T^*\in w^*\hbox{-}\Gamma_N[\X^*]}.$ Therefore
$\Gamma_N[\X^*]\sse w^*\hbox{-}\Gamma_N[\X^*].$
\end{proof}

\vskip4pt
Let $\X$ be a normed space$.$ The annihilator of a nonempty set ${A\in\X}$
is the set ${A^\perp\in\X^*}$ given by
\begin{eqnarray*}
A^\perp
&\kern-6pt=\kern-6pt&
\big\{f\in\X^*\!:A\sse\N(f)\big\}
=\big\{f\in\X^*\!:f(x)=0\;\hbox{for every}\;x\in A\big\}                \\
&\kern-6pt=\kern-6pt&
\bigcap_{x\in A}\{f\in\X^*\!:{x\in\N(f)}\}
=\bigcap_{x\in A}\{x\}^\perp
=\big\{f\in\X^*\!:f(A)=\0\big\},
\end{eqnarray*}
which is a subspace of $\X^*\!$, where
$\{x\}^\perp\!=\!\{{f\in\X^*}\!:{x\in\N(f)}\}\!\sse\!\X^*\!$ for every
${x\in\X}.$ The pre-annihilator of a nonempty set ${B\in\X^*}$ is the set
${^\perp\!B\in\X}$ given by
$$
{^\perp\!B}=\bigcap_{f\in B}\N(f)
=\big\{x\in\X\!:f(x)=0\;\hbox{for every}\;f\in B\big\}
=\bigcap_{f\in B}{^\perp\{f\}},
$$
which is a subspace of $\X$, where ${^\perp\{f\}}={\N(f)\sse\X}$ for every
${f\in\X^*}.$ The following identities are trivially verified:
${\0^\perp\!=\X^*}\!$, $\,{\X^\perp\!=\0}$,
${^\perp\!\0=\X}$, ${^\perp\!\X^*\!=\0}$.

\vskip6pt
If $\M$ is a linear manifold of a normed space $\X$, then
(see, e.g., \cite[Theorem 4.6-A]{Tay}, \cite[Problem 4.63]{EOT})
$$
{^\perp(\M^\perp)}=\M^-
\quad\hbox{(closure in $\X$)}.                   \leqno\phantom{(0)}
$$
If\/ $\U$ is a linear manifold of $\X^*\!$, then (see, e.g.,
\cite[Theorem 4.6-B]{Tay}),
$$
\U^-\sse({^\perp\U})^\perp
\quad\hbox{(closure in $\X^*$)}.                 \leqno(1)
$$
A linear manifold $\U$ of $\X^*\!$ was called {\it saturated}\/ in
\cite{Tay} if ${\U\!=({^\perp\U})^\perp}\!.$ If $\U$ is saturated, then
$\U$ is closed in $\X^*$ ($\U=\U^-$), and the converse fails: a closed
linear manifold of $\X^*\!$ (i.e., a subspace of $\X^*\!$) may not be
saturated \cite[p.225]{Tay}$.$ But $\U$ is saturated if and only if it is
weakly* closed \cite[Theorem 4.62-A]{Tay}; that is, closed in the weak*
topology ${\sigma(\X^*\!,\X)}$ of $\X^*\!$ \cite[p.209]{Tay})$.$ In general
this cannot be straightened to closedness in the weak topology
${\sigma(\X^*\!,\X^{**})}$ of $\X^*\!$ (which is stronger than
${\sigma(\X^*\!,\X)}$), since closedness of $\U$ in the weak topology
${\sigma(\X^*\!,\X^{**})}$ of $\X^*\!$ coincides with closed\-ness of $\U$
in the norm topology of $\X^*\!$ (Proposition 4.1)$.$ Summing up
(see also \cite[Corollary 16.5]{BP}):
$$
\U=({^\perp\U})^\perp
\;\;\iff\;\;
\U\;\hbox{is weak* closed in $\X^*$}.            \leqno(2)
$$
If $\M$ and $\N$ are subspaces of a Banach space $\X$,
then \cite[Theorem IV.4.8]{Kat}
$$
(\M+\N)^\perp=\M^\perp\!\cap\N^\perp,            \leqno(3)
$$
$$
\M^\perp\!+\N^\perp
\;\;\hbox{is closed in $\X^*$}
\;\;\iff\;\;
\M+\N
\;\;\hbox{is closed in $\X$},                    \leqno(4)
$$
\vskip4pt\noi
and if ${\M+\N}$ is closed in $\X$, then
$$
\M^\perp\!+\N^\perp=(\M\cap\N)^\perp.            \leqno(5)
$$

\vskip0pt\noi
\begin{lemma}
Let $\X$ be a Banach space.
\vskip2pt\noi
\begin{description}
\item{$\kern-4pt$\rm(a)}
If\/ $\M$ is a complemented subspace of\/ $\X$, then\/ $\M^\perp\!$ is a
complemented subspace of\/ $\X^*\!.$ Moreover, if\/ ${\N\sse\X}$ is a
complement of\/ ${\M\sse\X}$, then\/ ${\N^\perp\sse\X^*}$ is a complement
of\/ ${\M^\perp\sse\X^*}$.
\vskip2pt
\item{$\kern-4pt$\rm(b)}
If\/ $\U$ is an weak* complemented subspace of\/ $\X^*\!$, then\/
${^\perp\U}$ is a complemented subspace of\/ $\X.$ Also, if\/
${\V\sse\X^*}$ is a complement of\/ ${\U\sse\X^*\!}$, then\/
${{^\perp\V}\sse\X}$ is a complement of\/ ${{^\perp\U}\sse\X}$.
\end{description}
\end{lemma}

\begin{proof}
(a) If $\M$ is complemented subspace of a normed space $\X$, then there is
a subspace $\N$ of $\X$ for which
$$
\M+\N=\X
\quad\hbox{and}\quad
\M\cap\N=\0.
$$
Thus ${\M+\N}$ is closed in $\X.$ If $\X$ is a Banach space, then
${\M^\perp\!+\N^\perp}$ is closed in $\X^*$ by (4), and by (3,5)
$$
\M^\perp\!+\N^\perp=(\M\cap\N)^\perp=\0^\perp=\X^*\!,
$$
$$
\M^\perp\!\cap\N^\perp=(\M+\N)^\perp=\X^\perp=\0.
$$
Therefore the subspace $\N^\perp$ is a complement of $\M^\perp\!$, and so
the subspace $\M^\perp$ is complemented in $\X^*$.

\vskip6pt\noi
(b) Let $\X^*$ be the dual space of the normed space $\X$, which is itself
a Banach space$.$ Suppose $\U$ is complemented in $\X^*\!.$ Thus there is
a subspace $\V$ of $\X^*$ for which
$$
\U+\V=\X^*
\quad\hbox{and}\quad
\U\cap\V=\0.
$$
So according to (1) and (5) we get
$$
\X^*=\U+\V\sse({^\perp\U})^\perp+({^\perp\V})^\perp
=({^\perp\U}\,\cap\,{^\perp\V})^\perp.
$$
Hence
\vskip-4pt\noi
$$
^\perp\U\,\cap\,^\perp\V=\0.
$$
Now suppose $\U$ is weak* complemented$.$ Then $\U$ is weakly*
closed in $\X^*\!$ and we can take the subspace $\V$ weakly* closed
in $\X^*$ as well$.$ From (2) we get
$$
({^\perp\U})^\perp=\U
\quad\;\hbox{and}\;\quad
({^\perp\V})^\perp=\V,
$$
\vskip-4pt\noi
and from and (3),
$$
\0=\U\cap\V=({^\perp\U})^\perp\cap({^\perp\V})^\perp
=({^\perp\U}\,+{^\perp\V})^\perp.
$$
\vskip-2pt\noi
Hence
\vskip-6pt\noi
$$
{^\perp\U}\,+{^\perp\V}=\X,
$$
which concludes the proof$:$ the subspace $^\perp\U$ is complemented
in $\X$, where ${^\perp\V}$ is a complement of it.
\end{proof}

\vskip4pt
As it is well known (see, e.g,
\cite[p.9]{Aie}, \cite[p.135]{Heu}, \cite[Theorems 4.6-C,D,F,G]{Tay})
$$
{\R(T)^-}^\perp\!=\N(T^*),                       \leqno(6)
$$
$$
\R(T)^-\!={^\perp\N(T^*)},                       \leqno(7)
$$
$$
{^\perp\R(T^*)^-}\!=\N(T),                       \leqno(8)
$$
$$
\R(T^*)^-\!\sse\N(T)^\perp\!,                    \leqno\phantom{(0)}
$$
where the above inclusion becomes an identity if $\R(T)$ is closed in $\X$
(equivalently, if $\R(T^*)$ is closed in $\X^*$)
\cite[Theorem IV.5.13]{Kat}:
$$
\R(T^*)^-\!=\N(T)^\perp
\qquad\hbox{if $\,\R(T^*)\,$ is closed}.         \leqno(9)
$$

\vskip6pt\noi
{\bf Proof of Theorem 3.1.}
Let $T$ be an operator on a Banach space $\X$.
\vskip6pt\noi
If ${T\in\Gamma_R[\X]}$ (i.e., if $\R(T)^-$ is complemented in $\X$), then
by Lemma 4.2(a) $\R(T)^{-\perp}$ is complemented in $\X^*\!.$ Since
${\R(T)^{-\perp}\!=\N(T^*)}$ by (6), it follows that ${\N(T^*)}$ is
complemented in $\X^*\!$, which means ${T\in\Gamma_N[\X^*]}.$ Therefore
$$
T\in\Gamma_R[\X]
\quad\;\limply\quad
T^*\in\Gamma_N[\X^*].                            \eqno(\rm{a}_1)
$$
On the other hand, if $\X$ is reflexive, then by Lemma 4.1(b)
$$
T^*\in\Gamma_N[\X^*]
\quad\limply\quad
T^*\in w^*\hbox{-}\Gamma_N[\X^*].
$$
But if ${T^*\in w^*\hbox{-}\Gamma_N[\X^*]}$ (i.e., if $\N(T^*)$ is weak*
complemented in $\X^*\!$), then ${^\perp\!\N(T^*})$ is complemented in $\X$
according to Lemma 4.2(b)$.$ Since $\R(T)^-\!={^\perp\!\N(T^*)}$ by (7),
then $\R(T)^-$ is complemented in $\X$, which means ${T\in\Gamma_R[\X]}.$
Therefore
$$
T^*\in w^*\hbox{-}\Gamma_N[\X^*]
\quad\;\limply\quad
T\in\Gamma_R[\X].
$$
Hence if $\X$ is reflexive
$$
T^*\in\Gamma_N[\X^*]
\quad\;\limply\quad
T\in\Gamma_R[\X].                                 \eqno(\rm{a}_2)
$$
\vskip2pt\noi
With the assumption that $\X$ is reflexive still in force we get by
Lemma 4.1(a)
$$
T^*\in\Gamma_R[\X^*]
\quad\limply\quad
T^*\in w^*\hbox{-}\Gamma_R[\X^*].
$$
However, if ${T^*\!\in w^*\hbox{-}\Gamma_R[\X^*]}$ (i.e., if $\R(T^*)^-\!$
is weak* complemented in $\X^*\!$), then
${^\perp\R(T^*)^-}$ is complemented in $\X$ by Lemma 4.2(b)$.$ But
${\N(T)={^\perp\R(T^*)^-}}$ according to (8), and so $\N(T)$ is
complemented in $\X$, which means ${T\in\Gamma_N[\X]}.$ Thus
$$
T^*\in w^*\hbox{-}\Gamma_R[\X^*]
\quad\;\limply\quad
T\in\Gamma_N[\X].
$$
Hence if $\X$ is reflexive
$$
T^*\in\Gamma_R[\X^*]
\quad\;\limply\quad
T\in\Gamma_N[\X].                                  \eqno(\rm{b}_1)
$$
On the other hand, if ${T\in\Gamma_N[\X]}$ (i.e., if $\N(T)$ is
complemented in $\X$), then $\N(T)^\perp$ is complemented in $\X^*\!$ by
Lemma 4.2(a)$.$ If $\R(T)$ is closed, then $\R(T^*)=\R(T^*)^-\!=\N(T)^\perp$
by Proposition 4.5(a,b) and (9)$.$ Therefore ${\R(T^*)}$ is closed
and complemented in $\X^*$ (i.e., $\R(T^*)=\R(T^*)^-\kern-1pt$ and
$T^*\kern-1pt\in\Gamma_R[\X^*]\,).$ Thus
$$
\kern-7pt
\R(T)=\R(T)^-
\;\hbox{\rm and}\;\;
T\in\Gamma_N[\X]
\;\;\limply\;
\R(T^*)=\R(T^*)^-
\;\hbox{\rm and}\;\;
T^*\in\Gamma_R[\X^*].\!\!                          \eqno(\rm{b}_2)
$$
The converse follows from item (b$_1$), if $\X$ is reflexive, and
Proposition 4.5(a,b).\hfill$\qed$

\vskip0pt\noi
\section{Applications}

This section deals with the class ${\Gamma_R[\X]\cap\Gamma_N[\X]}$ of all
operators $T$ on a normed space $\X$ for which $\R(T)^-$ and $\N(T)$ are
complemented$.$ (Operators with this property are sometimes called inner
regular \cite[Section 0]{DK} --- see also \cite[Theorem 3.8.2]{Har}.)

\vskip6pt
Under reasonable conditions, compact operators and their normed-space
adjoints have complemented (closure of) range and complemented kernel
(Corollary 5.1), as it is the case for compact perturbations of bounded
below operators (Corollary 5.2), which also holds for semi-Fredholm
operators (Corollaries 5.3 and 5.4).

\vskip2pt\noi
\begin{corollary}
If\/ ${T\in\BX}$ is a compact operator on a reflexive Banach space\/ $\X$
with a Schauder basis, then
$$
T\in\Gamma_R[\X]\cap\Gamma_N[\X]
\quad\;\hbox{and}\;\quad
T^*\!\in\Gamma_R[\X^*]\cap\Gamma_N[\X^*].
$$
\end{corollary}

\begin{proof}
Suppose ${T\!\in\BX}$ is compact and $\X$ has a Schauder basis$.$ Thus there
is a sequence $\{T_n\}$ of finite-rank operators ${T_n\kern-1pt\in\BX}$ such
that ${T_n\uconv T}$ (i.e., $\{T_n\}$ converges uniformly, which means in
the operator norm topology of $\BX$, to $T$) and
$\R(T_n)\sse\R(T_{n+1})\sse\R(T)^-\!$ (see, e.g.,
\cite[Problem 4.58]{EOT})$.$ Since each $T_n$ is finite-rank (i.e.,
${\dim(\R(T_n))\kern-.5pt<\kern-1pt\infty}$), we get $\R(T_n)=\R(T_n)^-\!$.
Moreover, finite-dimensional subspaces of a Banach space are complemented
(see, e.g., \cite[Theorem A.1.25(i)]{Mul})$.$ Then
${T_n\kern-1pt\in\Gamma_R[\X]}$; equivalently, there exist continuous
projections ${E_n\kern-1pt\in\BX}$ and
${I\kern-1pt-\kern-1ptE_n\kern-1pt\in\BX}$ (and so with closed ranges) for
which $\R(E_n)=\R(T_n)$ --- cf$.$ Proposition 4.2(b), where $\{\R(E_n)\}$ is
an increasing sequence of subspaces$.$ Since $\{\R(E_n)\}$ is a monotone
sequence of subspaces, $\lim_n\R(E_n)$ exists in the following sense:
$$
\lim_n\R(E_n)=\bigcap_{n\ge1}\bigvee_{k\ge n}\R(E_k)
=\Big(\kern-1pt\bigcup_{n\ge1}\!\bigcap_{k\ge n}\R(E_k)\Big)^{_-}\!,
$$
where $\bigvee_{k\ge n}\!\R(E_k)$ is the closure of the span of
$\{\kern1pt\bigcup_{k\ge n}\!\R(E_k)\}$ (cf$.$ \cite[Definition 1]{CF})$.$
Thus concerning the complementary projections ${I\kern-1pt-\kern-1ptE_n}$,
$\lim_n\R({I\kern-1pt-\kern-1ptE_n})$ also exists$.$ Moreover,
$\lim_n\R(E_n)$ is a subspace of $\X$ included in $\R(T)^-\!$
($\lim_n\R(E_n)\sse\R(T)^-\!$ because $\R(T_n)\sse\R(T)^-).$ Since $\X$ has
a Schauder basis and $T$ is compact, the sequence of operators $\{E_n\}$
converges strongly (see, e.g., \cite[Hint to Problem 4.58]{EOT}) and so
$\{E_n\}$ is a bounded sequence$.$ Thus since (i) $\X$ is reflexive, (ii)
${T_n\sconv T}$ (i.e., $\{T_n\}$ converges strongly because it converges
uniformly), (iii) ${T_n\in\Gamma_R[\X]}$, (iv) $\R(E_n)=\R(T_n)$, (v)
$\{\|E_n\|\}$ is bounded, (vi) ${\lim_n\R(T_n)\sse\R(T)^-}\!$, and (vii)
$\lim_n\R({I\kern-1pt-\kern-1ptE_n})$ exists, then
$$
T\in\Gamma_R[\X]
$$
\cite[Theorem 2]{CF}$.$ Hence $\R(T)^-$ is complemented if $T$ is compact
and $\X$ is reflexive with a Schauder basis. Since reflexivity for $\X$ is
equivalent to reflexivity for $\X^*$ (Proposition 4.4(a,b)), since $\X^*$ has
a Schauder basis whenever $\X$ has (see, e.g., \cite[Theorem 4.4.1]{Meg}),
and since $T$ is compact if and only if $T^*$ is compact (see, e.g.,
\cite[Theorem 3.4.15]{Meg}), then the closure of the range of the compact
$T^*$ on $\X^*$, viz., $\R(T^*)^-\!$, is also complemented:
$$
T^*\in\Gamma_R[\X^*].
$$
Therefore, since $\X$ is reflexive, Theorem 3.1(a$_1$,b$_1$) ensures
$$
T\in\Gamma_R[\X]\;\limply T^*\in\Gamma_N[\X^*]
\quad\;\hbox{and}\;\quad
T^*\in\Gamma_R[\X^*]\;\limply T\in\Gamma_N[\X].
$$
\vskip-16pt\noi
\end{proof}

\vskip0pt\noi
\begin{proposition}
Let\/ $\X$ be Banach spaces and take\/ ${T\in\BX}$. The following
assertions are pairwise equivalent\/.
\vskip2pt\noi
\begin{description}
\item{$\kern-4pt$\rm(a)}
There exists\/ ${T^{-1}\!\in\B[\R(T),\X]}$ $\,\;(T$ has a bounded inverse
on its range\/$)$.
\vskip2pt
\item{$\kern-4pt$\rm(b)}
$\,T$ is bounded below $\,\;($there is an ${\alpha\kern-1pt>\kern-1pt0}$
such that ${\alpha\|x\|\le\|Tx\|}$ for all $x\in\kern-1pt\X)$.
\vskip2pt
\item{$\kern-4pt$\rm(c)}
$\,\N(T)=\0$ and\/ $\R(T)^-\kern-1pt=\R(T)$ $\,\;(T$ is
injective and has a closed range$)$.
\end{description}
\end{proposition}

\begin{proof}
See, e.g., \cite[Corollary 4.24]{EOT} or \cite[Theorem 1.2]{ST}.
\end{proof}

\vskip0pt\noi
\begin{corollary}
Let\/ ${T,K\!\in\BX}$ be operators on a Banach space\/ $\X.$
If\/ ${T\in\Gamma_R[\X]}$ is bounded below and $K$ is compact, then
$$
T+K\in\Gamma_R[\X]\cap\Gamma_N[\X]
\quad\;\hbox{and}\;\quad
T^*+K^*\!\in\Gamma_R[\X^*]\cap\Gamma_N[\X^*],
$$
and\/ $\R({T+K})$ is closed in\/ $\X$, and so is\/ $\R({T^*+K^*})$ in\/
$\X^*$.
\end{corollary}

\begin{proof}
We split the proof into two parts.

\vskip6pt\noi
{\bf Part 1}$.$
Suppose ${T\in\BX}$ is bounded below$.$ Equivalently, suppose ${T\in\BX}$ is
injective and has closed range (Proposition 5.1)$.$ Hence ${T\in\Phi_+[\X]}$
(Definition 2.1) --- as ${\dim\N(T)=0}$ whenever $T$ is injective$.$
In addition suppose ${T\in\Gamma_R[\X]}.$ Thus
${T\in\Phi_+[\X]\cap\Gamma_R[\X]}=\F_\ell[\X]$ (Proposition 3.1)$.$ But
$\F_\ell[\X]$ is invariant under compact perturbation by its very
definition$.$ (In fact, the collection of all compact operators
comprises an ideal of $\BX$ so that ${T+K\in\F_\ell[\X]}$ whenever
${T\in\F_\ell[\X]}$ and ${K\in\BX}$ is compact by Definition 2.2)$.$ Hence
${T+K\in\Phi_+[\X]\cap\Gamma_R[\X]}=\F_\ell[\X]$ (Proposition 3.1),
and so $\R({T+K})$ is closed (by Definition 2.1)$.$ Therefore
$$
\hbox{${T\in\Gamma_R[\X]}$ bounded below and $K$ is compact}
\quad\limply
$$
\vskip-2pt\noi
$$
T+K\in\Gamma_R[\X]
\;\;\hbox{ and $\;\;\R(T+K)$ is closed in the norm topology of $\X$}.
$$
\vskip2pt\noi
For a more elaborated proof under the same hypothesis, without using
Proposition 3.1, see \cite[Theorem 2]{Hol} (also see \cite[Theorem 2]{Gol}
and \cite[Lemma 3.1]{Cro})$.$

\vskip6pt\noi
{\bf Part 2}$.$
Actually, more is true$.$ Since finite-dimensional subspaces of a Banach
space are complemented (see, e.g., \cite[Theorem A.1.25(i)]{Mul}),
$$
\Phi_+[\X]\sse\Gamma_N[\X]
$$
by Definition 2.1$.$ Thus
$\F_\ell[\X]={\Phi_+[\X]\cap\Gamma_R[\X]}\sse{\Gamma_N[\X]\cap\Gamma_R[\X]}$
according to Proposition 3.1$.$ Since (as we saw above)
${T+K\in\F_\ell[\X]}$, we get
$$
T+K\in\Gamma_N[\X]\cap\Gamma_R[\X],
$$
and so (since $\R(T+K)$ is closed) Theorem 3.1(a$_1$,b$_2$) ensures
$$
(T+K)^*\in\Gamma_R[\X^*]\cap\Gamma_N[\X^*].
$$
Moreover, since $\R(T+K)$ is closed,
$$
\hbox{$\R((T+K)^*)$ is closed in the norm topology of $\X^*$}
$$
by Proposition 4.5$.$ Finally recall$:$ ${(T+K)^*=T^*+K^*}$.
\end{proof}

\vskip0pt\noi
\begin{corollary}
Let\/ ${T,K\!\in\BX}$ be operators on a Banach space\/ $\X.$ If\/ $T$ is a
left semi-Fredholm and\/ $K$ is a compact, then
$$
T+K\in\Gamma_R[\X]\cap\Gamma_N[\X]
\quad\;\hbox{and}\;\quad
T^*+K^*\!\in\Gamma_R[\X^*]\cap\Gamma_N[\X^*],
$$
both with closed range.
\end{corollary}

\begin{proof}
The assumption $T$ is injective was used in the proof of Corollary 5.2 only
to ensure ${\dim\N(T)<\infty}.$ So Corollary 5.2 can be promptly extended to
``$T$ has closed complemented range and finite-dimensional kernel'' (i.e.,
${T\in\Phi_+[\X]\cap\Gamma_R[\X]}=\F_\ell[\X]$) instead of ``$T$ has
closed complemented range and is injective'' (i.e., instead of assuming
``${T\in\Gamma_R[\X]}$ is bounded below'').
\end{proof}

\vskip0pt\noi
\begin{corollary}
Let\/ ${T,K\!\in\BX}$ be operators on a Banach space\/ $\X.$ If\/ $T$ is a
right semi-Fredholm and\/ $K$ is a compact, then
$$
T+K\in\Gamma_R[\X]\cap\Gamma_N[\X]
\quad\;\hbox{and}\;\quad
T^*+K^*\!\in\Gamma_R[\X^*]\cap\Gamma_N[\X^*],
$$
both with closed range.
\end{corollary}

\begin{proof}
If ${T\in\Phi_-[\X]}$, then by Definition 2.1 $\R(T)$ is closed and
${\codim\R(T)<\infty}$ (i.e., ${\dim\X/\R(T)<\infty}$), and so the subspace
$\R(T)$ is naturally complemented (in fact, if ${\codim\R(T)<\infty}$, then
every algebraic complement of $\R(T)$ is finite dimensional, thus
complemented --- since finite-dimensional subspaces are complemented ---
see e.g., \cite[Theorem A.1.25(i,ii)]{Mul})$.$ Hence
$$
\Phi_-[\X]\sse\Gamma_R[\X].
$$
Thus (Proposition 3.1),
$\F_r[\X]={\Phi_-[\X]\cap\Gamma_N[\X]}\sse{\Gamma_R[\X]\cap\Gamma_N[\X]}.$
But, $\F_r[\X]$ is also invariant under compact perturbation by its own
definition$.$ (Indeed, since compact operators form an ideal of $\BX$,
we get ${T+K\in\F_r[\X]}$ whenever ${T\in\F_r[\X]}$ and ${K\in\BX}$
is compact according to Definition 2.2)$.$ Therefore, by Proposition 3.1,
${T+K\in\Phi_-[\X]\cap\Gamma_N[\X]}=\F_r[\X]$, and hence $\R({T+K})$ is
closed by Definition 2.1$.$ Thus, by the above displayed inclusion,
$$
{T+K\in\F_r[\X]}
={\Phi_-[\X]\cap\Gamma_N[\X]}\sse\Gamma_R[\X]\cap\Gamma_N[\X]
$$
for every compact $K$ in $\BX$ and every ${T\in\F_r[\X]}$ (hence
$\R({T+K})$ is closed, and so is $\R(({T+K})^*)$)$.$ Since $\R(T)$ is
closed whenever ${T\in\F_r[\X]}$, it follows by Theorem 3.1(a$_1$,b$_2$)
that ${(T+K)^*\in\Gamma_R[\X^*]\cap\Gamma_N[\X^*]}$.
\end{proof}

\vskip0pt\noi
\begin{remark}
(a) Corollary 5.2 can be trivially restricted to ``${T\in\BX}$ is
invertible'' (i.e., ${T\in\BX}$ has a bounded inverse --- an inverse in
$\BX$) instead of ``${T\in\Gamma_R[\X]}$ is bounded below''.
\vskip4pt\noi
{\narrower\narrower\it
If\/ $T$ is invertible and\/ $K$ is compact, both acting on a Banach
space\/ $\X$, then\/ ${T+K\in\Gamma_R[\X]\cap\Gamma_N[\X]}$ and\/
${T^*\kern-1pt+K^*\!\in\Gamma_R[\X^*]\cap\Gamma_N[\X^*]}$, both with
closed range.
\vskip5pt}\noi

\vskip0pt\noi
(b) $\!$An $\kern-1pt$operator $\kern-1pt$is $\kern-1pt$semi-Fredholm
$\kern-1pt$if $\kern-1pt$it $\kern-1pt$is $\kern-1pt$left $\kern-1pt$or
$\kern-1pt$right $\kern-1pt$semi-Fredholm$.$ $\!$According $\kern-1pt$to
$\kern-1pt$Corollaries 5.3 and 5.4 compact $\kern-1pt$perturbations
$\kern-1pt$of $\kern-1pt$semi-Fredholm $\kern-1pt$operators $\kern-1pt$have
$\kern-1pt$com\-plemented range and kernel for the operator itself and for
its normed-space adjoint.
\vskip4pt\noi
{\narrower\narrower\it
If\/ $T$ is semi-Fredholm and\/ $K$ is compact, both acting on a Banach
space\/ $\X$, then\/ ${T+K\in\Gamma_R[\X]\cap\Gamma_N[\X]}$ and\/
${T^*\kern-1pt+K^*\!\in\Gamma_R[\X^*]\cap\Gamma_N[\X^*]}$, both with closed
range
\vskip5pt}\noi
In particular, Fredholm operators (i.e., operators in ${\F_\ell\cap\F_r}$),
and so essentially invertible operators $T\!$, have complemented range and
kernel for both $T$ and $T^*\!.$ It is worth noticing (as implicitly
embedded in the proofs of Corollaries 5.2 and 5.4) that the conclusion
${T+K\in\Gamma_R[\X]\cap\Gamma_N[\X]}$ for a semi-Fredholm operator $T$
(i.e., if ${T\in\F_\ell\cup\F_r}$) is readily verified from Proposition 3.1
by using the inclusions ${\Phi_+[\X]\sse\Gamma_N[\X]}$ and
${\Phi_-[\X]\sse\Gamma_R[\X]}$ (as in the proofs of Corollaries 5.2 and
5.4), and also by the fact that compact perturbations of semi-Fredholm
operators are clearly semi-Fredholm (see, e.g., \cite[Theorem 5.6]{ST})$.$
More along this line in \cite{DK}..

\vskip6pt\noi
(c) According to item (a) or item (b), an invertible operator $T$ has
complemented range and kernel for $T$ and $T^*\!.$ Since the collection of
all invertible operators (i.e., of all operators with a bounded inverse) on
a Banach space $\X$ is an open group included in $\BX$, then the sets
$\Gamma_R[\X]\cap\Gamma_N[\X]$ and $\Gamma_R[\X^*]\cap\Gamma_N[\X^*]$ are
algebraically and topologically large$.$ It has been asked in
\cite[Question 3.1]{KD2} whether the sets $\Gamma_R[\X]$ and $\Gamma_N[\X]$
(and consequently, the sets $\Gamma_R[\X^*]$ and $\Gamma_N[\X^*]$) are open
in $\BX$ (or in $\B[\X^*]$) --- i.e., whether they are open in the operator
norm topology.
\end{remark}

\vskip-10pt\noi
\bibliographystyle{amsplain}

\end{document}